\documentclass[12pt]{amsart}
\usepackage{amssymb}

\newcommand{\eps}{\varepsilon}

\newcommand{\N}{{\mathbb N}}
\newcommand{\C}{{\mathbb C}}

\newcommand{\R}{{\mathbb R}}

\newcommand{\tef}{transcendental entire function}

\newcommand{\mconn}{multiply connected}

\theoremstyle{plain}
\newtheorem{theorem}{Theorem}[section]
\newtheorem{corollary}[theorem]{Corollary}

\newtheorem{lemma}[theorem]{Lemma}
\theoremstyle{definition}

\theoremstyle{remark}

\theoremstyle{problem}

\theoremstyle{example}
\newtheorem{example}[theorem]{Example}

\begin{document}


\title[Regularity and fast escaping points of entire functions]{Regularity and fast escaping points of entire functions}

\author{P. J. Rippon}
\address{Department of Mathematics and Statistics \\
The Open University \\
   Walton Hall\\
   Milton Keynes MK7 6AA\\
   UK}
\email{p.j.rippon@open.ac.uk}

\author{G. M. Stallard}
\address{Department of Mathematics and Statistics \\
The Open University \\
   Walton Hall\\
   Milton Keynes MK7 6AA\\
   UK}
\email{g.m.stallard@open.ac.uk}

\thanks{Both authors were supported by the EPSRC grant EP/H006591/1.}




\begin{abstract}
Let $f$ be a {\tef}. The fast escaping set $A(f)$, various regularity conditions on the growth of the maximum modulus of~$f$, and also, more recently, the quite fast escaping set $Q(f)$ have all been used to make progress on fundamental questions concerning the iteration of~$f$. In this paper we establish new relationships between these three concepts.

We prove that a certain weak regularity condition is necessary and sufficient for $Q(f)=A(f)$ and give examples of functions for which $Q(f)\ne A(f)$.

We also apply a result of Beurling that relates the size of the minimum modulus of $f$ to the growth of its maximum modulus in order to establish that a stronger regularity condition called log-regularity holds for a large class of functions, in particular for functions in the Eremenko-Lyubich class ${\mathcal B}$.

\end{abstract}

\maketitle

\section{Introduction}
\setcounter{equation}{0}
Let $f:\C\to \C$ be a {\tef} and denote by $f^{n},\,n=0,1,2,\ldots\,$, the $n$th iterate of~$f$. The {\it Fatou set} $F(f)$ is the set of points $z \in \C$
such that $(f^{n})_{n \in \N}$ forms a normal family in some neighborhood of $z$.  The complement of $F(f)$ is called the {\it Julia set} $J(f)$ of $f$. An introduction to the properties of these sets can be found in~\cite{wB93}.

In recent years there has been much interest in understanding the structure of the {\it escaping set\,} $I(f)$ of $f$, defined as follows:
\[
I(f)=\{z:f^n(z)\to \infty\;\text{ as }n\to\infty\}.
\]
The first general results about $I(f)$ were obtained by Eremenko in~\cite{E}, who showed that $I(f)\cap J(f)\neq \emptyset$ and $J(f)=\partial I(f)$, and also that all the components of $\overline{I(f)}$ are unbounded. Much of this work has been directed towards solving what is known as Eremenko's conjecture, which states that all the components of $I(f)$ itself are unbounded. Though this conjecture is still unsolved, much progress has been made in special cases (see~\cite{RRRS} for example), and in the general case it is known \cite{RS05} that $I(f)$ has at least one unbounded component.

This result in the general case was proved by using the {\it fast escaping set\,} $A(f)$, which consists of points whose iterates tend to $\infty$ as fast as possible. The set $A(f)$ was introduced in~\cite{BH99} and can be defined as follows; see~\cite{RS09a}. Let
\[
M(r)=M(r,f) = \max_{|z|=r} |f(z)|, \;\; r > 0,
\]
and put
\[
A(f) = \{z: \mbox{there exists } \ell \in \N \text{ such that }
|f^{n+\ell}(z)| \geq M^n(R), \text{ for } n \in \N\},
\]
where $M^n(r)$ denotes iteration of $M(r)$, and $R>0$ is any value such that $M(r)>r$ for $r\geq R$. The set $A(f)$ has many strong properties that can be used in the study of $I(f)$ and $J(f)$; see~\cite{RS09a}. For example, all the components of $A(f)$ are unbounded.

In~\cite{RS09a} we gave a weaker condition on the iterates of a point that is sufficient to ensure that the point is in $A(f)$. To be precise we showed that if $\mu(r) = \eps M(r)$, $r>0$, where $\eps \in (0,1)$,  then
\begin{equation}\label{fast-weak}
A(f)= \{z: \mbox{there exists } \ell \in \N \text{ such that }
|f^{n+\ell}(z)| \geq \mu^n(R), \text{ for } n \in \N\},
\end{equation}
where $R>0$ is sufficiently large to ensure that $\mu(r) > r$, for $r \geq R$.

Sixsmith~\cite[Theorem~2]{dS12} recently gave a weaker but somewhat more involved condition than the one in~\eqref{fast-weak} which ensures that points are in $A(f)$.

A further weakening of~\eqref{fast-weak} is to replace $\eps M(r)$ by $M(r)^{\eps}$. In other words, for $0<\eps<1$, we let
\[
\mu_{\eps}(r)=M(r)^{\eps}, \;\; r > 0,
\]
and consider the set
\[
Q_{\eps}(f) = \{z: \mbox{there exists } \ell \in \N \text{ such that }
|f^{n+\ell}(z)| \geq \mu_{\eps}^n(R), \text{ for } n \in \N\},
\]
where $R>0$ has the property that $\mu_{\eps}(r)>r$ for $r\geq R$. It is easy to check that the definition of $Q_{\eps}(f)$ is independent of the choice of $R$ with the property that $\mu_{\eps}(r)>r$ for $r\geq R$.

Points that lie in $Q_{\eps}(f)$ for some $\eps$, $0<\eps<1$, have played a significant role in recent work related to both Eremenko's conjecture and a conjecture of Baker that all {\tef}s of sufficiently small order have no unbounded Fatou components (see~\cite{RS11b}), and also in papers related to the Hausdorff measure and Hausdorff dimension of various Julia and escaping sets (see the remark following Theorem~\ref{classB-thm}). Because of the increased significance of these points we define the set
\[
Q(f)=\bigcup_{0<\eps<1}Q_{\eps}(f)
\]
to be the {\it quite fast escaping set} of $f$. It follows immediately from these definitions that
\begin{equation}\label{AQI}
A(f)\subset Q_{\eps}(f)\subset Q(f)\subset I(f), \;\text{ for } 0<\eps<1.
\end{equation}

The central question in this paper concerns the relationship between $Q(f)$ and $A(f)$, and we show that for many transcendental entire functions $Q(f)$ is identical to $A(f)$. However, we also give examples of entire functions for which $Q(f)\ne A(f)$; further such examples, which are related to Eremenko's conjecture and Baker's conjecture and are much more complicated, are given in~\cite{RS11}.

An important special case of our results is the following theorem concerning the much-studied Eremenko-Lyubich class ${\mathcal B}$, which consists of {\tef}s whose set of singular values (that is, critical values and asymptotic values) is bounded.
\begin{theorem}\label{classB-thm}
Let $f$ be a {\tef} in the class ${\mathcal B}$. Then $Q(f)=A(f)$.
\end{theorem}

{\it Remark.} Theorem~\ref{classB-thm} implies that some results concerning the Hausdorff measure and Hausdorff dimension of the escaping sets and Julia sets of certain functions in the class ${\mathcal B}$ also hold for the fast escaping sets of such functions. In particular, \cite[Theorem 1.1]{BKS} and \cite[Theorem 1.1]{jP11} were proved by considering points in $Q(f)$ (although this is not stated explicitly) and so, since the functions in those results are in the class ${\mathcal B}$, it follows from Theorem~\ref{classB-thm} that the conclusions of those results also hold for $A(f)$. \\

The proof of Theorem~\ref{classB-thm} is in two steps. First we introduce a regularity condition on $f$ which implies that $Q(f)=A(f)$. Let $R>0$ be any value such that $M(r)>r$ for $r\geq R$. We say that $f$ is
\begin{itemize}
\item {\it ${\eps}$-regular}, where $0<\eps<1$, if there exists $r=r(R)>0$ such that
\begin{equation}\label{eps-reg}
\mu_{\eps}^n(r)\ge M^n(R),\;\;\text{for } n\in \N,
\end{equation}
or, equivalently, if there exists $\ell=\ell(R)\in\N$ such that
\[
\mu_{\eps}^{n+\ell}(R)\ge M^n(R),\;\;\text{for } n\in \N;
\]
\item {\it weakly regular} if $f$ is ${\eps}$-regular for all $0<\eps<1$.
\end{itemize}

We note that the definition of $\eps$-regularity is independent of the value of $R$ with the property that $M(r)>r$ for $r\geq R$.

Roughly speaking, the aim of these regularity conditions, and other related conditions to be introduced later in the paper, is to prevent the maximum modulus of a {\tef} from behaving on long intervals in a way that is too similar to that of a polynomial.

We show that, not only is weak regularity a sufficient condition for $Q(f)=A(f)$, but it is also a necessary condition for $Q(f)=A(f)$.
\begin{theorem}\label{main2a}
Let $f$ be a {\tef} and $0<\eps<1$. Then
\begin{itemize}
\item[(a)]
$f$ is ${\eps}$-regular if and only if\, $Q_{\eps}(f)=A(f)$;
\item[(b)]
$f$ is weakly regular if and only if\, $Q(f)=A(f)$.
\end{itemize}
\end{theorem}

The proof that weak regularity is a sufficient condition for $Q(f) = A(f)$ is straightforward. However, in order to prove that it is also a necessary condition we require a rather general result from \cite{RS13a} that guarantees the existence of points with a prescribed rate of escape; see Section~\ref{mainthm}.

The second (and main) step in the proof of Theorem~\ref{classB-thm} is to show that if $f$ is in the class~${\mathcal B}$, then~$f$ is weakly regular. Recall that if~$f$ is in the class ${\mathcal B}$, then~$f$ is bounded on some path to $\infty$; see \cite[page~993]{EL92}. Our next result shows that~$f$ is weakly regular whenever~$f$ is bounded on a path to~$\infty$ and indeed under a much weaker hypothesis involving the {\it minimum modulus} of~$f$, which is defined as follows:
\[
m(r)=m(r,f)=\min_{|z|=r}|f(z)|,\;\;r>0.
\]
Thus Theorem~\ref{classB-thm} can be viewed as a special case of the following result.

\begin{theorem}\label{main1a}
Let $f$ be a {\tef} such that, for some $r(f)> 1$,
\begin{equation}\label{minmod1}
m(r)\le M(r)^{1-K/\log r},\;\;\text{for }r\ge r(f),
\end{equation}
where $K=4\log 4$. Then $f$ is weakly regular, and hence $Q(f)=A(f)$.
\end{theorem}

In fact we show that~\eqref{minmod1} implies that $f$ satisfies a stronger regularity condition called `log-regularity' that we define in Section~\ref{regprops}. This log-regularity condition was introduced by Anderson and Hinkkanen in \cite{AH99} in connection with Baker's conjecture, and it is easier to check than weak regularity. We prove Theorem~\ref{main1a} by using a comparatively unknown result of Beurling from his thesis~\cite{aB33}, which relates the size of the minimum modulus of a function to the growth of its maximum modulus.

The structure of the paper is as follows. We begin in Section~\ref{basics} by discussing some of the basic properties of $Q(f)$. In Section~\ref{mainthm} we prove Theorem~\ref{main2a} and also construct functions for which $Q(f)\neq A(f)$.

In Section~\ref{regprops}, we discuss several regularity conditions for a {\tef} and the relationships between them; in particular, we show that log-regularity implies weak regularity and that $\eps$-regularity is equivalent to a type of regularity introduced in \cite{RS09a}. Then in Section~\ref{logreg} we prove Theorem~\ref{main1a}. Finally, in Section~\ref{example} we construct examples in order to show that $\eps$-regularity, weak-regularity and log-regularity are not equivalent conditions.

\section{Basic properties of $Q(f)$}\label{basics}
\setcounter{equation}{0}
In this section we show that $Q(f)$ has many of the same basic properties as $I(f)$ and $A(f)$. As mentioned earlier, Eremenko~\cite{E} established the following properties of $I(f)$:
\begin{equation}\label{Iprops}
I(f)\neq \emptyset,\;\; I(f)\cap J(f)\neq\emptyset,\;\;J(f)=\partial I(f),
\end{equation}
and he also showed that $\overline{I(f)}$ has no bounded components. The fast escaping set also has the properties listed in~\eqref{Iprops}, and in addition $A(f)$ itself has no bounded components, as noted earlier; see~\cite{BH99} and~\cite{RS05}.

The sets $Q(f)$ and $Q_{\eps}(f)$ have similar properties to $I(f)$. The proofs of these properties are similar to those for $I(f)$ so we give only brief details for $Q(f)$.

\begin{theorem}\label{Qprops}
Let $f$ be a {\tef}. Then
\begin{equation}\label{Qprops1}
Q(f)\neq \emptyset,\;\; Q(f)\cap J(f)\neq\emptyset,\;\;J(f)=\overline{Q(f)\cap J(f)},\;\;J(f)=\partial Q(f),
\end{equation}
and $\overline{Q(f)}$ has no bounded components. Similar properties hold for each set $Q_{\eps}(f)$, where $0<\eps<1$.
\end{theorem}
\begin{proof}
The first three properties in \eqref{Qprops1} are immediate since these properties hold for $A(f)$ (see~\cite[Theorem~5.1]{RS09a}) and $A(f)\subset Q(f)$.

Because $Q(f)$ is infinite and completely invariant under $f$, we have $J(f)\subset \overline{Q(f)}$, and this implies that $J(f)\subset \partial Q(f)$ since any open subset of $Q(f)$ is contained in $F(f)$. Also, no point of $\partial Q(f)$ lies in $F(f)$ since any such point would have a neighbourhood in $Q(f)$ by an application of the distortion theorem for iterates in Fatou components; see~\cite[Lemma~7]{wB93}. Hence $J(f)=\partial Q(f)$.

Finally, if $\overline{Q(f)}$ has a bounded component, $E$ say, then there is an open topological annulus $A$ lying in the complement of $\overline{Q(f)}$ that surrounds $E$. Since $\overline{Q(f)}$ is completely invariant under $f$, we deduce by Montel's theorem that $A\subset F(f)$ and since $J(f)=\partial Q(f)$ we deduce that $A$ is contained in a {\mconn} Fatou component. But any {\mconn} Fatou component of~$f$ is contained in $A(f)$ (see~\cite[Theorem~2]{RS05} or~\cite[Theorem~4.4]{RS09a}) and hence in $Q(f)$, so we obtain a contradiction. This completes the proof of Theorem~\ref{Qprops}.
\end{proof}
{\it Remark}\;\;In view of the considerable interest in Eremenko's conjecture, it is natural to ask whether all the components of $Q(f)$ are unbounded.

\section{Proof of Theorem~\ref{main2a}}\label{mainthm}
\setcounter{equation}{0}
In this section we prove Theorem~\ref{main2a} and also construct examples for which $Q(f)\ne A(f)$. The proof of Theorem~\ref{main2a} uses the following general result from \cite[Theorem~1.4]{RS13a}.

\begin{theorem}\label{rate}
Let $f$ be a {\tef}. There exists $R=R(f)>0$ with the property that whenever $(a_n)$ is a positive sequence such that
\begin{equation}\label{aprops}
a_n\ge R\;\:\text{and}\;\; a_{n+1}\le M(a_n),\;\text{ for }n\in\N,
\end{equation}
there exists a point $\zeta\in J(f)$ and a sequence $(n_j)$ with $n_j\to\infty$ as $j\to \infty$, such that
\begin{equation}\label{QnotA1}
|f^n(\zeta)| \geq a_n, \mbox{ for } n \in\N, \;\mbox{ but } |f^{n_j}(\zeta)| \le M^2(a_{n_j}), \mbox{ for } j \in \N.
\end{equation}
\end{theorem}
This result guarantees in particular that for any given rate of escape that is at most that of `fast escape' there exist points that escape at least at this given rate and in some sense at no faster rate.

\begin{proof}[Proof of Theorem~\ref{main2a}]
We first prove part~(a), which states that $f$ is $\eps$-regular if and only if $Q_{\eps}(f)=A(f)$. Let $0<\eps<1$ and let $R>0$ be so large that $\mu_{\eps}(r)>r$ for $r\ge R$ and so that Theorem~\ref{rate} can be applied with this value of $R$.

Suppose first that $f$ is ${\eps}$-regular. If $z\in Q_{\eps}(f)$, then there exists $\ell\in \N$ such that
\[
|f^{n+\ell}(z)|\ge \mu_{\eps}^n(R),\;\; \text{for }n\in\N.
\]
Let $r=r(R)$ be as in the definition of ${\eps}$-regularity. Then there exists $m\in N$ such that $\mu^m_{\eps}(R)>r$ so
\[
|f^{n+\ell+m}(z)|\ge \mu_{\eps}^{n+m}(R)\ge \mu_{\eps}^n(r)\ge M^n(R),\;\; \text{for }n\in\N,
\]
and hence $z\in A(f)$. Thus $Q_{\eps}(f)=A(f)$.

The proof of part~(a) in the opposite direction uses Theorem~\ref{rate}. The idea is to show that if the function $f$ is not $\eps$-regular, where $0<\eps<1$, then the set $Q_{\eps}(f)\setminus A(f)$ is non-empty.

If $f$ is not $\eps$-regular, then for each $\ell \in \N$ there exists $n(\ell)\in \N$ such that
\begin{equation}\label{not-weak-reg}
\mu_{\eps}^{n+\ell}(R) < M^n(R),\;\;\text{for } n\ge n(\ell).
\end{equation}
Next, by Theorem~\ref{rate}, with $a_n=\mu_{\eps}^n(R)$, $n\in \N$, there exists a point~$\zeta$ and a sequence $(n_j)$ with $n_j\to\infty$ as $j\to \infty$, such that
\begin{equation}\label{QnotA1}
|f^n(\zeta)| \geq \mu_{\eps}^n(R), \mbox{ for } n \in \N,
\end{equation}
and
\begin{equation}\label{QnotA2}
 |f^{n_j}(\zeta)| \le M^2(\mu_{\eps}^{n_j}(R)), \mbox{ for } j \in \N.
\end{equation}

It follows from \eqref{QnotA1} that $\zeta\in Q_{\eps}(f)$. Also,~\eqref{QnotA2} and~\eqref{not-weak-reg} together imply that, for each $\ell \in \N$, we have
\[
|f^{(n_j-\ell+2)+\ell-2}(\zeta)| = |f^{n_j}(\zeta)| \le M^2(\mu_{\eps}^{n_j}(R)) < M^2(M^{n_j-\ell}(R)) = M^{n_j-\ell+2}(R),
\]
for sufficiently large values of $j$. Hence $\zeta \notin A(f)$, so $Q_{\eps}(f)\setminus A(f) \neq \emptyset$, as required.

Part~(b) of Theorem~\ref{main2a} follows easily from part~(a). If~$f$ is weakly regular, then $f$ is $\eps$-regular for all $0<\eps<1$, so $Q_{\eps}(f)=A(f)$ for all $0<\eps<1$, by part~(a), and hence $Q(f)=A(f)$. In the other direction, if $Q(f)=A(f)$, then $Q_{\eps}(f)=A(f)$ for all $0<\eps<1$, by \eqref{AQI}, so $f$ is $\eps$-regular for all $0<\eps<1$, by part~(a), and hence $f$ is weakly regular.
\end{proof}

Next we show that there exist {\tef}s $f$ such that $Q(f)\ne A(f)$.

\begin{theorem}\label{product}
Let $0<\eps<1$. There exists a {\tef} $f$ such that $Q_{\eps}(f)\ne A(f)$, and hence $Q(f) \ne A(f)$.
\end{theorem}

We use the following lemma.

\begin{lemma}\label{exlem}
Let
\begin{equation}\label{prod}
 f(z) = \prod_{m=1}^{\infty}\left( 1 - \frac{z}{r_m}\right), \mbox{ where } r_m > 0 \mbox{ for } m \in \N.
\end{equation}
Then there exist $\delta \in (0,1)$ and a sequence $(c_m)$, with $c_m \to 0$ as $m \to \infty$ and $0<c_m<1$ for $m\in\N$, such that, if $r_{m+1} > 4r_m$ for $m \in \N$, then
\begin{equation}\label{poly}
\delta c_m |z|^m < |f(z)| < c_m |z|^m,
\;\mbox{ for } 2r_m < |z| < \tfrac{1}{2}r_{m+1},\; m\in\N.
\end{equation}
\end{lemma}
\begin{proof}
This follows easily by writing
\[
f(z) = \prod_{k=1}^m \frac{z}{r_k} \left(\frac{r_k}{z} - 1 \right)
\prod_{k=m+1}^{\infty}\left( 1 - \frac{z}{r_k} \right).\qedhere
\]
\end{proof}
\begin{proof}[Proof of Theorem~\ref{product}]
We now indicate how Lemma~\ref{exlem} can be used to show that, if $0<\eps<1$ and $r_{m+1}/r_m$ is sufficiently large for each $m \in \N$, then the function $f$ given by \eqref{prod} is not $\eps$-regular and hence $Q_{\eps}(f)\ne A(f)$.

Let $\eps' \in (\eps,1)$. It follows from~\eqref{poly} that, for each $m \in \N$, if $r_{m+1}/r_m$ is sufficiently large, then there exists $R_m > r_m$ such that
\begin{equation}\label{Mnu}
 M(r) > r^{\eps' m} \mbox{ and } \mu_{\eps}(r) = M(r)^{\eps} < r^{\eps m}, \mbox{ for } R_m \le r < \tfrac{1}{2}r_{m+1}.
\end{equation}
Now fix $m \in \N$ and set $M_m(r) = r^{\eps' m}$ and $\mu_{\eps,m}(r) = r^{\eps m}$. Since $\eps < \eps'$, there exists $N_m \in \N$ such that
$(\eps m)^{N_m +m} < (\eps' m)^{N_m}$ and hence, for $r>0$,
\begin{equation}\label{Nn}
\mu_{\eps,m}^{N_m+m}(r) < M_m^{N_m}(r).
 \end{equation}
So, if we choose $r_{m+1}$ sufficiently large to ensure that $M^{N_m}(R_m) < \tfrac{1}{2}r_{m+1}$, then it follows from~\eqref{Mnu} and~\eqref{Nn}, with $r=R_m$, that
\begin{equation}\label{long}
 \mu_{\eps}^{N_m+m}(R_m) < M^{N_m}(R_m), \mbox{ for } m \in \N.
\end{equation}

Now let $R$ be so large that $M(r)>r$, for $r\ge R$ and let $m\in\N$ be such that $R_m > M(R) > \mu_{\eps}(R)$. Then there exists $N'_m \in \N$ such that
\begin{equation}\label{Nm}
 \mu_{\eps}^{N'_m}(R) \leq R_m < \mu_{\eps}^{N'_m+1}(R).
\end{equation}
It follows from~\eqref{Nm} and~\eqref{long} that
\begin{eqnarray*}
\mu_{\eps}^{N_m+m+N'_m}(R) & \leq & \mu_{\eps}^{N_m+m}(R_m) < M^{N_m}(R_m)\\
& < & M^{N_m}(\mu_{\eps}^{N'_m+1}(R)) < M^{N_m+N'_m+1}(R),
\end{eqnarray*}
and so, if $n$ is sufficiently large,
\begin{equation}\label{m}
 \mu_{\eps}^{n+m-1}(R) < M^n(R).
\end{equation}
Hence $f$ is not $\eps$-regular, as required.
\end{proof}
{\it Remark}\;\;In Section~\ref{example} we show by a different method that if $0<\eps<\eps'<1$, then there exists a function that is $\eps'$-regular but not $\eps$-regular, thus giving an alternative construction of a function for which $Q(f) \ne A(f)$.

\section{Regularity conditions}\label{regprops}
\setcounter{equation}{0}
Theorem~\ref{main2a} tells us that weak-regularity implies that $Q(f)=A(f)$. Unfortunately, weak-regularity is not straightforward to check directly. In this section we discuss several other regularity conditions and the relationships between these conditions and weak regularity. In particular, we discuss log-regularity which implies weak regularity and is often easy to check.

We say that a {\tef} $f$ is {\it log-regular} if there exists $c>0$ such that the function $\phi(t) = \log M(e^t)$ satisfies
\begin{equation}
\frac{\phi'(t)}{\phi(t)}\ge \frac{1+c}{t}\,,\;\;\text{for large } t.
\end{equation}

As mentioned earlier, the log-regularity condition was introduced by Anderson and Hinkkanen in~\cite{AH99} in relation to Baker's conjecture. The name log-regular was suggested by Aimo Hinkkanen in a private communication and used first by Sixsmith~\cite{dS11}. It was pointed out in \cite[page~205]{H} that this condition holds for all {\tef}s of finite order and positive lower order; recall that the {\it order} $\rho$ and the {\it lower order} $\lambda$ of a {\tef} $f$ are:
\[
\rho=\limsup_{r\to\infty} \frac{\log\log M(r)}{\log r}\;\;\text{and}\;\; \lambda=\liminf_{r\to\infty} \frac{\log\log M(r)}{\log r}.
\]
The paper ~\cite{dS11} gives many other classes of entire functions that are log-regular; for example, any composition of {\tef}s is log-regular if one of the functions in the composition is log-regular and any {\tef}~$f$ is log-regular if and only if its derivative is log-regular.

There are also many {\tef}s which are not log-regular. For example, if $f$ has a {\mconn} Fatou component, then $f$ is not log-regular. This follows from the fact that such Fatou components contain large annuli within which the iterates of $f$ behave like the iterated maximum modulus of~$f$ (see, for example, \cite[Theorem~1.2 and Lemma~2.1]{BRS11}); combined with log-regularity this fact would contradict the distortion theorem for iterates in a Fatou component~\cite[Lemma~7]{wB93}.

We now show that log-regularity implies weak-regularity. This result can also be deduced by combining results given in~\cite{RS08} and~\cite{RS09a} (see Theorem~\ref{condns} part~(d)) but we include a direct proof here for completeness.

\begin{theorem}\label{main2}
Let $f$ be a {\tef}. If $f$ is log-regular, then $f$ is weakly regular.
\end{theorem}

{\it Remark}\; In Section~\ref{example} we give an example of an entire function which is weakly regular but not log-regular.

To prove Theorem~\ref{main2}, we first give a result on convex functions that leads to an alternative characterisation of log-regularity. This alternative characterisation is used again in Sections~\ref{logreg} and~\ref{example} of this paper to determine whether certain functions are or are not log-regular, and also in~\cite{dS11}.
\begin{theorem}\label{log-reg-phi}
Let $\phi$ be an increasing convex function defined on $\R$. The following are equivalent:
\begin{itemize}
\item[(a)]
there exist $t_0>0$ and $c>0$ such that
\begin{equation}
\frac{\phi'(t)}{\phi(t)}\ge \frac{1+c}{t},\;\;\text{for } t\ge t_0;
\end{equation}
\item[(b)]
there exist $t_0>0$ and $c>0$ such that for all $k>1$ and $d=k^c$, we have
\[\phi(kt)\ge kd\phi(t),\;\;\text{for } t\ge t_0;\]
\item[(c)]
there exist $t_1>0$, $k>1$ and $d>1$, such that
\[\phi(kt)\ge kd\phi(t),\;\;\text{for } t\ge t_1.\]
\end{itemize}
\end{theorem}
{\it Remark}\;\; For definiteness, in the statement of this theorem $\phi'(t)$ denotes the right derivative of $\phi$ at $t$.\\

To obtain our alternative characterisation of log-regularity we apply Theorem~\ref{log-reg-phi} with $\phi(t)=\log M(e^t)$.
\begin{corollary}\label{log-reg}
Let $f$ be a {\tef}. The following are equivalent:
\begin{itemize}
\item[(a)]
$f$ is log-regular;
\item[(b)]
there exist $r_0>0$ and $c>0$ such that for all $k>1$ and $d=k^c$, we have
\[M(r^k)\ge M(r)^{kd},\;\;\text{for } r\ge r_0;\]
\item[(c)]
there exist $r_1>0$, $k>1$ and $d>1$, such that
\[M(r^k)\ge M(r)^{kd},\;\;\text{for } r\ge r_1.\]
\end{itemize}
\end{corollary}
{\it Remarks}\;\;1. Corollary~\ref{log-reg} part~(b) shows that log-regularity represents a strengthening of the following version of Hadamard convexity, given in~\cite[Lemma 2.2]{RS08}: for every {\tef} $f$ there exists $r(f)>0$ such that, for all $r \geq r(f)$ and all $k>1$,
\[
M(r^k) \ge M(r)^k.
\]

2. The condition in Corollary~\ref{log-reg} part~(c) is equivalent to a regularity condition used by Wang in~\cite{W}, namely, that there exists $k>1$ such that
\[
\liminf_{r\to\infty}\frac{\log M(r^k,f)}{\log M(r,f)}>k.
\]
Thus the latter condition is equivalent to log-regularity.\\

\begin{proof}[Proof of Theorem~\ref{log-reg-phi}]
To prove that~(a) implies~(b), we note that, for $t\ge t_0$,
\[
\log\frac{\phi(kt)}{\phi(t)}\ge \int_t^{kt}\frac{\phi'(u)}{\phi(u)}\,du\ge \int_t^{kt}\frac{1+c}{u}\,du=\log k^{1+c}.
\]
Next, it is clear that~(b) implies~(c).

To prove that~(c) implies~(a), suppose that $t\ge kt_1$. Since $\phi'$ is increasing,
\[
t(1-1/k)\phi'(t)\ge \int_{t/k}^t \phi'(u)\,du=\phi(t)-\phi(t/k),
\]
so, using the fact that $\phi(t)\ge kd\phi(t/k)$, we have
\[
\frac{t\phi'(t)}{\phi(t)}\ge\frac{1-\phi(t/k)/\phi(t)}{1-1/k}\ge \frac{1-1/(kd)}{1-1/k}=1+c,\;\;\text{for }t\ge kt_1,
\]
where $c>0$, as required.
\end{proof}

Now we deduce Theorem~\ref{main2}.
\begin{proof}[Proof of Theorem~\ref{main2}]
Suppose that $f$ is log-regular with constant $c>0$ and let~$\eps$ be given with $0<\eps<1$.

Let $R$ be so large that $M(r)>r$, for $r\ge R$. By Corollary~\ref{log-reg} part~(b), there exists $r_0\ge R$ such that for all $k>1$ and $d=k^c$, we have
\[
M(r^k)\ge M(r)^{kd},\;\;\text{for } r\ge r_0.
\]
We apply this estimate with $d=1/\eps>1$ and $k=d^{1/c}>1$. Then
\[
\mu_{\eps}(r^k)\ge M(r)^k,\;\;\text{for } r\ge r_0.
\]
Thus, by induction,
\[
\mu_{\eps}^n(r_0^k)\ge (M^n(r_0))^k \ge M^n(R),\;\;\text{for } n\ge 0.
\]
Hence $f$ is $\eps$-regular.
\end{proof}

We point out that there are other types of regularity conditions that are closely related to log-regularity and weak-regularity. These were introduced in~\cite{RS08} and~\cite{RS09a} in relation to Baker's conjecture and the existence of a spider's web structure for the fast escaping set. We summarize the relationships between these conditions in the following theorem which shows that there are many examples of {\tef}s that are weakly regular.

\begin{theorem}\label{condns}
Let $f$ be a {\tef} and let $R>0$ be such that $M(r)>r$, for $r\ge R$. Then
\begin{itemize}
\item[(a)]
$f$ is $\eps$-regular, where $0<\eps<1$, if and only if there is a sequence $(r_n)$ such that
\begin{equation}\label{weak}
r_n\ge M^n(R)\quad\text{and}\quad M(r_n)\ge r_{n+1}^m,\;\;\text{for } n\ge 0,
\end{equation}
where $m=1/\eps$;
\item[(b)]
if $m>1$ and there exists a real function $\psi$ defined on $[r_0, \infty)$, where $r_0>0$, such that
\begin{equation}\label{psireg}
\psi(r) \geq r\;\;\text{and}\;\;  M(\psi(r)) \geq \left(\psi(M(r))\right)^m,\;\;\text{for }r\geq r_0,
\end{equation}
then there is a sequence $(r_n)$ such that~\eqref{weak} holds, so~$f$ is $\eps$-regular with $\eps=1/m$;
\item[(c)] if $f$ is of finite order and there exist $n\in\N$ and $0<q<1$ such that
\[
M(r)\ge \exp^{n+1}((\log^n r)^q), \;\;\text{for large } r,
\]
then, for each $m>1$, there is a real function $\psi$ such that \eqref{psireg} holds, so $f$ is weakly regular;
\item[(d)] if $f$ is log-regular, then for each $m>1$ there is a real function $\psi$ such that~\eqref{psireg} holds, so $f$ is weakly regular.
\end{itemize}
\end{theorem}
\begin{proof}
To prove part~(a) we first note that the condition \eqref{eps-reg} implies that \eqref{weak} holds with $m=1/\eps$ and $r_n=\mu_{\eps}^n(r)$, $n\ge 0$. On the other hand, if \eqref{weak} holds and $\eps=1/m$, then
\[
\mu_{\eps}(r_n)\ge r_{n+1},\;\;\text{for }n\ge 0,
\]
so
\[
\mu^n_{\eps}(r_0)\ge r_n\ge M^n(R),\;\;\text{for }n\ge 0,
\]
and hence \eqref{eps-reg} holds.

Part~(b) follows from the fact that if $f$ satisfies the condition \eqref{psireg}, for some value of $m$, then~\eqref{weak} holds for the same value of $m$, by taking $r_n= \psi(M^n(R))$.

Part~(c) is \cite[Theorem~6]{RS08}. There the function $f$ was assumed to have order less than $1/2$, but the proof actually required only finite order.

Part~(d) was also proved in \cite{RS08}, by using a similar argument to the one given in the first part of the proof of Lemma~\ref{log-reg}. In \cite[Section~7]{RS08} we showed that if $f$ is log-regular, with constant $c>0$, then for all $m>1$ the statement~\eqref{psireg} holds with $\psi(r)=r^k$, where $k=m^{1/c}$.
\end{proof}

{\it Remarks}\;\;1. Theorem~\ref{condns} part~(a) shows that $\eps$-regularity is equivalent to the type of regularity introduced in \cite[Corollary 8.3 part~(b)]{RS09a} with $m=1/\eps$.

2. A similar regularity condition to that in Theorem~\ref{condns} part~(b) played a key role in~\cite{dS11}. There the additional mild assumption was made that the function $\psi$ is increasing and the name {\it $\psi$-regularity} was used.

3. In Section~\ref{example} of this paper we use Theorem~\ref{condns} part~(c) to show that one of our examples is weakly regular.\\

Finally we mention two other regularity conditions. First, the condition
\[
\log M(2r)\ge d\log M(r),\;\;\text{for large } r,
\]
where $d>1$, was mentioned in~\cite{wB12} in relation to a result on the packing dimension of $I(f)\cap J(f)$; see the remark after the proof of Theorem~\ref{log-reg-suff} of this paper. This regularity condition is easily seen to imply log-regularity. Next, the regularity condition
\begin{equation}\label{G-reg}
\frac{\log M(2r)}{\log M(r)}\to c\;\;\text{as }r\to \infty,
\end{equation}
where $c\ge 1$, was used in~\cite{gS93} in relation to Baker's conjecture. This condition again implies log-regularity in the case $c>1$, but it does not imply log-regularity in the case $c=1$. Indeed using an approximation method (see Lemma~\ref{clunie-kovari} of this paper) we can construct a {\tef} such that
\[
\log M(r)\sim \log r\log\log r\;\;\text{as }r\to\infty,
\]
and so~$f$ satisfies~\eqref{G-reg} with $c=1$ but~$f$ is not log-regular.

\section{Sufficient conditions for log-regularity}\label{logreg}
\setcounter{equation}{0}
First we prove the following sufficient condition for log-regularity, which is rather general; in particular, it does not require any hypothesis about the order of the function. Taken together with Theorem~\ref{main2}, this result implies Theorem~\ref{main1a} and hence Theorem~\ref{classB-thm}.

\begin{theorem}\label{log-reg-suff}
Let $f$ be a {\tef} such that, for some $r(f)>1$,
\begin{equation}\label{minmod2}
m(r)\le M(r)^{1-K/\log r},\;\;\text{for }r\ge r(f),
\end{equation}
where $K=4\log 4$. Then $f$ is log-regular.

In particular, if $f$ is in the class ${\mathcal B}$, then $f$ is log-regular.
\end{theorem}
We use the following result of Beurling \cite[page~96]{aB33}, which also played a key role in~\cite{RS11}. Here, for any subset $E$ of $(0,\infty)$ we denote by $m_{\ell}(E)$ the logarithmic measure of $E$; that is,
\[
m_\ell(E)=\int_E \frac{dt}{t}.
\]
\begin{lemma}\label{Beur}
Let $f$ be analytic in $\{z:|z| < r_0\}$, let $0\le r_1<r_2< r_0$, and put
\[
E=\{t\in (r_1,r_2):m(t)\le \mu\}, \;\;\text{where } 0<\mu<M(r_1).
\]
Then
\begin{equation}\label{estimate}
\log \frac{M(r_2)}{\mu}>c\exp\left(\frac12 m_{\ell}(E)\right)\log \frac{M(r_1)}{\mu},
\end{equation}
where $c=\pi/(4\sqrt 2)$.
\end{lemma}
We remark that this result actually holds for all values of $\mu$ such that $0<\mu<M(r_2)$ since for $M(r_1)\le \mu<M(r_2)$ the right-hand side of \eqref{estimate} is negative or~$0$.
\begin{proof}[Proof of Theorem~\ref{log-reg-suff}] By Corollary~\ref{log-reg} part~(c), it is sufficient to show that there exist $k>1$ and $d>1$ such that
\begin{equation}\label{reg}
\log M(r^k)\ge kd\log M(r), \;\;\text{for large }r.
\end{equation}
The idea of the proof is to use the estimate \eqref{minmod2} together with Lemma~\ref{Beur} to show that, for sufficiently large~$r$, the maximum modulus of $f$ must increase by a certain amount over each of a sequence of relatively short adjacent intervals of the form $[r,\lambda r]$, where $\lambda>1$, and then to combine these increases to obtain~\eqref{reg} for suitable values of $k$ and $d$.

First put
\[
\delta(r)=\frac{K}{\log r},\quad r\ge r_0(f),
\]
where $r_0(f)\ge r(f)$ is so large that $\delta(r)\le 1$ for $r\ge r_0(f)$. Now take $k>1$ and~$\lambda$ such that $1<\lambda<r^{k-1}$ for $r\ge r_0(f)$,
and apply Lemma~\ref{Beur} with
\[
r_1=s,\;\;r_2=\lambda s,\;\;\mu=M(\lambda s)^{1-\delta(r^k)},\;\;\text{where }r\le s<\lambda s\le r^k.
\]
This gives
\begin{eqnarray*}
\log M(\lambda s)^{\delta(r^k)}
&>& c\exp\left(\frac12\int_s^{\lambda s} \frac{dt}{t}\right)\log \frac{M(s)}{M(\lambda s)^{1-\delta(r^k)}}\\
&=&c \lambda^{1/2}\left(\log M(s)-(1-\delta(r^k))\log M(\lambda s)\right),
\end{eqnarray*}
so
\[
\log M(\lambda s)>\frac{1}{(c\lambda^{1/2})^{-1}\delta(r^k)+1-\delta(r^k)}\log M(s).
\]

If we now take
\begin{equation}\label{lambda-def}
\lambda=\frac{16}{(1-\delta(r^k))^2},
\end{equation}
then $1<\lambda<r^{k-1}$ for $r\ge r_1(f,k)\ge r_0(f)$, say. Therefore, for $r\ge r_1(f,k)$ and $r\le s<\lambda s\le r^k$, we have
\begin{eqnarray}\label{s-est}
\log M(\lambda s)
&>&\frac{1}{(4c)^{-1}\delta(r^k)(1-\delta(r^k))+1-\delta(r^k)}\log M(s)\notag\\
&>&\frac{1}{1-a\delta(r^k)}\log M(s),
\end{eqnarray}
where $a=1-(4c)^{-1}$. Note that $1/2<a<1$, since $c=\pi/(4\sqrt 2)>1/2$.

Then, for $r\ge r_1(f,k)$, we define $n=n(r,k)\in\N$ by the condition
\begin{equation}\label{n-def}
\lambda^n r\le r^k<\lambda^{n+1}r,
\end{equation}
and apply \eqref{s-est} to $s=\lambda^m r$, for $m=0,1,\ldots, n-1$. Combining these estimates gives
\begin{equation}\label{lower-est}
\log M(r^k)>\left(\frac{1}{1-a\delta(r^k)}\right)^n \log M(r).
\end{equation}
By the right-hand inequality in \eqref{n-def},
\[
\lambda^n>r^{(k-1)n/(n+1)},
\]
so \eqref{lower-est} gives
\begin{equation}\label{p-def}
\log M(r^k)>\left(\lambda^{1/2}\right)^{np(r)} \log M(r) > \left(r^{(k-1)/2}\right)^{p(r)n/((n+1)} \log M(r),
\end{equation}
where
\[
p(r)=\frac{\log (1/(1-a\delta(r^k)))}{\log \lambda^{1/2}}=\frac{\log (1/(1-a\delta(r^k)))}{\log (4/(1-\delta(r^k)))}\ge \frac{a\delta(r^k)}{\log (4/(1-\delta(r^k)))},
\]
by \eqref{lambda-def}.

Thus to deduce \eqref{reg} from \eqref{p-def}, it is sufficient to choose $k>1$, $d>1$ so that for sufficiently large $r$ we have
\[
\left(\frac{k-1}{2}\right)\left(\frac{n}{n+1}\right)\left(\frac{a\delta(r^k)}{\log (4/(1-\delta(r^k)))}\right)\log r\ge \log kd;
\]
that is, since $\delta(r)=K/\log r$,
\begin{equation}\label{suff}
\left(\frac{n}{n+1}\right)\left(\frac{2aK}{4\log (4/(1-\delta(r^k)))}\right)\ge \left(\frac{k}{k-1}\right)\log kd.
\end{equation}

To do this we first note that, for each fixed $k>1$,
\begin{equation}\label{delta-prop}
\delta(r^k) \to 0 \;\;\text{as }r \to \infty,
\end{equation}
so, by \eqref{lambda-def} and \eqref{n-def},
\begin{equation}\label{n-prop}
n=n(r,k) \to \infty \;\;\text{as }r \to \infty.
\end{equation}

Take $b=\frac12(1+2a)$ and note that $1<b<2a$, because $a>1/2$. Since
\[
k\mapsto \frac{k\log k}{k-1}\;\;\text{is strictly increasing on } (1,\infty),\;\;\text{with}\;\; \lim_{k\to 1}\frac{k\log k}{k-1}=1,
\]
and $b>1$, we can choose absolute constants $k>1$ and $d>1$ such that
\[
b>\left(\frac{k}{k-1}\right)\log kd.
\]

Thus, by \eqref{delta-prop} and \eqref{n-prop}, and the facts that $K=4\log 4$ and $b<2a$, we can choose $r_2(f) \ge r_1(f,k)$ such that \eqref{suff} holds for all $r\ge r_2(f)$ with these values of $k$ and $d$. This completes the proof of Theorem~\ref{log-reg-suff}.
\end{proof}

{\it Remark}\; A similar condition to~\eqref{minmod2} appears in a recent paper of Bergweiler~\cite{wB12}. The main result of \cite{wB12} is that if $f$ is a {\tef} with no {\mconn} Fatou components and
\begin{equation}\label{Berg1}
\frac{\log\log M(r)}{\log \log r}\to\infty\;\;\text{as }r\to\infty,
\end{equation}
then the packing dimension of $I(f)\cap J(f)=2$. Also, \cite[Corollary~4.1]{wB12} states that if
\begin{equation}\label{K(r)}
m(r)\le M(r)^{1-K(r)/\log r},\;\;\text{for large }r,
\end{equation}
where $K(r)\to\infty$ as $r\to \infty$, then $f$ has no {\mconn} Fatou components and \eqref{Berg1} holds.

The assertion that \eqref{K(r)} implies that $f$ has no {\mconn} Fatou components is justified in \cite{wB12} by using results from either \cite{jZ} or \cite{BRS11}, and  the assertion that~\eqref{K(r)} implies~\eqref{Berg1} is deduced from a result of Fenton~\cite{pF80}. We remark that this latter implication can also be deduced by using the argument from the proof of Theorem~\ref{log-reg-suff} above. Briefly, \eqref{reg} implies that
\[
\liminf_{r\to\infty}\frac{\log\log M(r)}{\log \log r} \ge \frac{\log kd}{\log k},
\]
and if $k=2$, say, while the constant $K$ in Theorem~\ref{log-reg-suff} is large, then \eqref{suff} can be satisfied by a correspondingly large value of $d$. See also the related discussion of the consequences of log-regularity in \cite[page~204]{H}.\\

We conclude this section by pointing out that Lemma~\ref{Beur} can be used to deduce other sufficient conditions for log-regularity, such as the following, in which the restriction on the minimum modulus holds only on a sufficiently large set.
\begin{theorem}\label{main3}
Let $f$ be a {\tef}. Suppose that there exist $\alpha$,~$\beta,~k$, with $0<\alpha<\beta<1$, $k>1$ and $r(f)>0$ such that the set
\[
F_r=\{\rho\in(r,r^k): m(\rho)\le M(r^k)^{\alpha/k}\},\;\;r>0,
\]
satisfies
\[
m_\ell(F_r)\ge 2\log\left(\frac{k-\alpha}{c(\beta-\alpha)}\right),\;\;\text{for } r\ge r(f),
\]
where $c=\pi/(4\sqrt 2)$ is the constant in Lemma~\ref{Beur}. Then $f$ is log-regular.
\end{theorem}

\begin{proof}
Let $d=1/\beta$. By Corollary~\ref{log-reg} part~(c), it is sufficient to show that
\begin{equation}\label{dvalue}
\log M(r^k)\ge kd \log M(r)= \frac{k}{\beta}\log M(r), \;\;\text{for } r\ge r(f).
\end{equation}
To do this we apply Lemma~\ref{Beur} with
\[
r_1=r,\;\;r_2=r^k\;\;\text{and}\;\;\mu=M(r^k)^{\alpha/k},
\]
where $r\ge r(f)$, which gives
\[
(1-\alpha/k)\log M(r^k)> \left(\log M(r)-(\alpha/k)\log M(r^k)\right)c\exp(m_\ell(F_r)/2);
\]
that is,
\[
\left(1-\frac{\alpha}{k}+\frac{c\alpha}{k}\exp(m_\ell(F_r)/2)\right)\log M(r^k)>c\exp(m_\ell(F_r)/2)\log M(r).
\]
Thus~\eqref{dvalue} holds if
\[
1-\frac{\alpha}{k}+\frac{c\alpha}{k}\exp(m_\ell(F_r)/2)\le \left(\frac{\beta}{k}\right)\,c\exp(m_\ell(F_r)/2),
\]
and this is equivalent to
\[
m_\ell(F_r)\ge 2\log\left(\frac{k-\alpha}{c(\beta-\alpha)}\right),\;\;\text{for } r\ge r(f),
\]
as required.
\end{proof}

\section{Examples}\label{example}
\setcounter{equation}{0}
In this section we construct examples to show that log-regularity, $\eps$-regularity and weak-regularity are not equivalent.

\begin{example}\label{ex1}
There exists a {\tef} that is weakly regular but not log-regular.
\end{example}

\begin{example}\label{ex2}
Suppose that $0<a<b<1$. Then there exists a {\tef}~$f$ such that
\begin{itemize}
\item[(a)] for all $\eps\in[b,1)$, $f$ is $\eps$-regular;
\item[(b)] for all $\eps\in (0,a]$, $f$ is not $\eps$-regular and hence $Q_{\eps}(f)\ne A(f)$.
\end{itemize}
It follows that $f$ is not weakly regular and $Q(f)\ne A(f)$.
\end{example}

We construct both examples using the following result of Clunie and Kovari~\cite{CK}.

\begin{lemma}\label{clunie-kovari}
Let $\Phi$ be a convex increasing function on $\R$ such that $\Phi(t)\neq O(t)$ as $t\to\infty$. Then there exists a {\tef} $f$ such that
\[
\log M(e^t,f)\sim \Phi(t) \;\text{ as } t\to \infty.
\]
\end{lemma}

Often we show that a function is $\eps$-regular or weakly regular by first showing that it is log-regular. For these two examples, however, this method is not available and instead we use Theorem~\ref{condns} part~(c) in Example~\ref{ex1} and a direct construction in Example~\ref{ex2}.
\begin{proof}[Proof of Example~\ref{ex1}]
The idea is to start by constructing a `model' function for $\log M(e^t,f)$, denoted by $\Phi(t)$, which is convex and satisfies the hypotheses of Theorem~\ref{condns} part~(c) but does not satisfy the log-regularity condition, and then apply Lemma~\ref{clunie-kovari} to $\Phi$ to obtain an entire function~$f$ which also satisfies the hypotheses of Theorem~\ref{condns} part~(c), and so is weakly regular, but is not log-regular.

First, we define
\[
\mu(t)=\exp(t^{1/2}), \;\;t\ge 0.
\]
Then take $t_0>1$ so large that $\exp(\tfrac34t^{1/2})>t$ for $t\ge t_0$, and define
\[
t_n=\mu^n(t_0)\;\;\text{and}\;\;k_n=t_{n+1}^{1/4},\;\;n\ge 0.
\]
Since $t_n\ge t_0$, $t_{n+1}^{3/4}>t_n$ and $k_n>1$, for $n\ge 0$, we have
\begin{equation}\label{kn1}
t_{n+1}>k_n t_n>t_n,\;\;\text{for }n\ge 0.
\end{equation}

We now define
\[
\Phi(t)=
\left\{\begin{array}{ll} \mu_n(t),\;& t\in[t_{n+1}/k_n,t_{n+1}],\; n\ge 0, \\
\mu(t), & \mbox{otherwise,}
\end{array}\right.
\]
where $\mu_n(t)$ denotes the linear function such that $\mu_n(t)=\mu(t)$ for $t=t_{n+1}/k_n$ and $t=t_{n+1}$. Note that the intervals $[t_{n+1}/k_n,t_{n+1}]$ do not overlap, by~\eqref{kn1}. We claim that this function $\Phi$ has the following properties:
\begin{itemize}
\item[(a)] $\Phi$ is convex;
\item[(b)] $\Phi$ satisfies
\[
\lim_{t\to\infty}\frac{\log \Phi(t)}{t}=0;
\]
\item[(c)] for any $k>1$, we have
\[
\frac{\Phi(t_{n+1})}{\Phi(t_{n+1}/k)}\to k\;\text{ as }n\to\infty.
\]
\end{itemize}

Property~(a) is clear since the function $\mu$ is convex. To prove property~(b), note that
\[
\limsup_{t\to\infty}\frac{\log \mu(t)}{t}=\limsup_{t\to\infty}t^{-1/2}=0
\]
and that
\[
\max\left\{\frac{\log \Phi(t)}{t}:t\in[t_{n+1}/k_n,t_{n+1}]\right\}\le \frac{\log \mu(t_{n+1})}{t_{n+1}/k_n}=\frac{1}{t_{n+1}^{1/4}}\to 0\;\text{ as }n\to \infty.
\]

Finally, property~(c) holds because if $n$ is so large that $k_n>k$, then
\begin{eqnarray*}
\frac{\Phi(t_{n+1})}{\Phi(t_{n+1}/k)}
&=&\dfrac{\mu(t_{n+1})(t_{n+1}-t_{n+1}/k_n)}
{\mu(t_{n+1}/k_n)(t_{n+1}-t_{n+1}/k)+\mu(t_{n+1})(t_{n+1}/k-t_{n+1}/k_n)}\\
&<&\dfrac{\mu(t_{n+1})(t_{n+1}-t_{n+1}/k_n)}
{\mu(t_{n+1})(t_{n+1}/k-t_{n+1}/k_n)}\\
&=& \dfrac{1-1/k_n}{1/k-1/k_n}\to k\;\text{ as }n\to\infty,
\end{eqnarray*}
as required.

Now we apply Lemma~\ref{clunie-kovari} to $\Phi$ to give a {\tef} $f$ such that
\begin{equation}\label{phi-nu}
\log M(e^t,f)= \Phi(t)(1+\eps(t)),
\end{equation}
where $\eps(t)\to 0$ as $t\to \infty$. Then
\[
\frac{\log\log M(e^t,f)}{t}=\frac{\log \Phi(t)}{t}+\frac{O(\eps(t))}{t}\to 0\;\;\text{as }t\to\infty,
\]
by property~(b), so $f$ has order~$0$.

Also,
\[
\log M(e^t,f)\ge \frac12 \Phi(t)\ge \frac12\exp(t^{1/2})\ge \exp(t^{1/4}),\;\;\text{for large }t,
\]
so $f$ satisfies the hypotheses of Theorem~\ref{condns} part~(c), with $p=1$ and $q=1/4$. Hence $f$ is weakly regular.

Finally, by~\eqref{phi-nu} and property~(c), we have, for any $k>1$,
\[
\frac{\log M(e^{t_{n+1}})}{\log M(e^{t_{n+1}/k})} \sim \frac{\Phi(t_{n+1})}{\Phi(t_{n+1}/k)}\to k \;\;\text{as }n\to\infty.
\]
Thus, if $r_n=e^{t_n}$, $n\in\N$, then
\[
M(r_{n+1},f)=M(r_{n+1}^{1/k})^{k+o(1)}\;\;\text{as }n\to\infty,
\]
so $f$ is not log-regular.
\end{proof}

\begin{proof}[Proof of Example~\ref{ex2}]
Recall that $0<a<b<1$. Once again the idea is to start by constructing a convex `model' function $\Phi(t)$ for $\log M(e^t,f)$. In this case, $\Phi$ is constructed to be $\beta$-regular but not $\alpha$-regular, for some $\alpha,\beta$ with $a<\alpha<\beta<b$. Then we apply Lemma~\ref{clunie-kovari} to $\Phi$ to obtain an entire function~$f$ which is $b$-regular but not $a$-regular.

First take $c\in (a,b)$ and construct a strictly increasing sequence $(t_n)$ as follows:
\begin{equation}\label{tn-def}
t_0=1,\;\;t_1=2,\;\;\frac{t_{n+2}}{t_{n+1}}=\left(\frac1c\right)\frac{t_{n+1}}{t_n},\;\;n \ge 0.
\end{equation}
Then let $\Phi$ be the function that satisfies
\begin{equation}\label{phi-def}
\Phi(0)=0,\;\;\Phi(t_n)=t_{n+1}, \;\;n\ge 0,
\end{equation}
and is linear on $(-\infty,t_0]$ and on each interval $[t_n,t_{n+1}]$, $n\ge 0$.

Next, for $0<s<1$, we let $\Psi_s(t)=s\Phi(t)$, $t\in\R$, and we estimate the size of $\Psi_s(t)$ for $t\in [t_n,t_{n+1}]$. A general point $t\in [t_n,t_{n+1}]$ is of the form
\begin{equation}\label{general-t}
t=(1-\lambda)t_{n}+\lambda t_{n+1},\;\;\text{where }0\le \lambda\le 1,
\end{equation}
and, by~\eqref{phi-def},
\begin{equation}\label{nu-t}
\Psi_s(t)=s(1-\lambda)t_{n+1}+s\lambda t_{n+2}.
\end{equation}
By~\eqref{general-t},
\[
\frac{t}{t_n}=1-\lambda+\lambda\frac{t_{n+1}}{t_n},
\]
and hence, by~\eqref{nu-t} and~\eqref{phi-def},
\begin{eqnarray*}
\frac{\Psi_s(t)}{t_{n+1}}
&=&s(1-\lambda)+s\lambda\frac{t_{n+2}}{t_{n+1}}\\
&=&s(1-\lambda)+\left(\frac{s\lambda}{c}\right)\frac{t_{n+1}}{t_{n}}\\
&=&s(1-\lambda)+\frac{s}{c}\left(\frac{t}{t_n}-(1-\lambda)\right)\\
&=&\left(\frac{s}{c}\right)\frac{t}{t_n}-(1-\lambda)s\left(\frac{1-c}{c}\right).
\end{eqnarray*}
Hence, for all $t\in [t_n,t_{n+1}]$ and $s\in (0,1)$, we have
\begin{equation}\label{lower}
\frac{\Psi_s(t)}{t_{n+1}}\ge \left(\frac{s}{c}\right)\frac{t}{t_n}-\left(\frac{1-c}{c}\right)
\end{equation}
and
\begin{equation}\label{upper}
\frac{\Psi_s(t)}{t_{n+1}}\le \left(\frac{s}{c}\right)\frac{t}{t_n}.
\end{equation}

To prove part~(a) it is sufficient to prove that $f$ is $b$-regular. Put $\beta=\tfrac12(c+b)$, so $c<\beta<b$. Then, by~\eqref{lower} with $s=\beta$ and $t\in[t_n,t_{n+1}]$, we have
\begin{equation}\label{lower-1}
\frac{\Psi_{\beta}(t)}{t_{n+1}}\ge \left(\frac{\beta}{c}\right)\frac{t}{t_n}-\left(\frac{1-c}{c}\right)\ge \left(\frac{\beta+c}{2c}\right)\frac{t}{t_n}>\frac{t}{t_n},
\end{equation}
provided that
\[
\frac{t}{t_n}\ge \frac{2(1-c)}{\beta-c}=\frac{4(1-c)}{b-c}.
\]
Thus if $N$ is so large that
\begin{equation}\label{N1}
\frac{t_{N+1}}{t_N}>\frac{4(1-c)}{b-c},
\end{equation}
and $t\in [(4(1-c)/(b-c))t_N,t_{N+1}]$, then, by repeated application of \eqref{lower-1}, we obtain
\[
\frac{\Psi_{\beta}^n(t)}{t_{N+n}}\ge \left(\frac{\beta+c}{2c}\right)^n\frac{t}{t_N},\;\; \text{for }\;n\in\N.
\]
In particular, if $t\in [(4(1-c)/(b-c))t_N,t_{N+1}]$, then
\begin{equation}\label{nu-n}
\Psi_{\beta}^{n}(t)\ge t_{N+n}, \;\;\text{for }\;n\in\N.
\end{equation}

Since $0<\beta/b<1$, we can apply Lemma~\ref{clunie-kovari} to obtain a {\tef} $f$ and $N\in\N$ such that
\begin{equation}\label{f-cond}
\left(\frac{\beta}{b}\right)\Phi(t)\le \log M(e^t,f)\le \Phi(t),\;\;\text{for }t\ge t_N,
\end{equation}
and also such that~\eqref{N1} holds, and hence~\eqref{nu-n}.

For $0<\eps<1$, we put
\begin{equation}\label{mu-def}
\mu_{\eps}(r)=M(r,f)^{\eps}\quad\text{and}\quad \psi_{\eps}(t)=\log \mu_{\eps}(e^t)=\log M(e^t,f)^{\eps}.
\end{equation}
Then, by~\eqref{f-cond},
\[
\psi_{b}(t)\ge \beta\Phi(t)=\Psi_{\beta}(t),\;\;\text{for } t\ge t_N.
\]
Therefore, if $t\in [(4(1-d)/(b-d))t_N,t_{N+1}]$, then, by~\eqref{nu-n} and~\eqref{f-cond} again,
\[
\psi_{b}^n(t)\ge \Psi_{\beta}^n(t)\ge t_{N+n}=\Phi^n(t_N)\ge \log M^n(e^{t_N},f),\;\;\text{for }n\in\N,
\]
so
\[
\mu_{b}^n(e^t)\ge M^n(e^{t_N},f),\;\;\text{for }n\in\N.
\]
Hence $f$ is $b$-regular.

To prove part~(b) it is sufficient to show that $f$ is not $a$-regular. Put $\alpha=\tfrac12(a+c)$, so $a<\alpha<c$. Then, by~\eqref{upper} with $s=\alpha$ and $t\in [t_n,t_{n+1}]$, we have
\begin{equation}\label{upper1}
\frac{\Psi_{\alpha}(t)}{t_{n+1}}\le \left(\frac{\alpha}{c}\right)\frac{t}{t_n},\;\;\text{for } n\in\N.
\end{equation}
We deduce, by~\eqref{f-cond},~\eqref{mu-def} and~\eqref{upper1}, that if $t\in [t_n,t_{n+1}]$, then
\[
\frac{\log \mu_{a}(e^t)}{t_{n+1}}=\frac{a\log M(e^t)}{t_{n+1}}\le \frac{a\Phi(t)}{t_{n+1}}=\left(\frac{a}{\alpha}\right)\frac{\Psi_{\alpha}(t)}{t_{n+1}}\le \left(\frac{a}{c}\right)\frac{t}{t_n},\;\;\text{for } n\in \N.\]
Since $0<a/c<1$, we deduce that for any $n\in \N$ and $t\in (t_n,t_{n+1}]$ there exists $k\in\N$ such that
\[
\frac{\log \mu_{a}^{k}(e^t)}{t_{n+k}}\le \left(\frac{a}{c}\right)^k\frac{t}{t_n}\le 1,\quad\text{so}\quad \log\mu_{a}^k(e^t)\le t_{n+k}=\Phi^k(t_n).
\]
By applying this result repeatedly we see that for any $n\in \N$, $t\in [t_n,t_{n+1}]$ and $\ell\in \N$ there exists $k\in\N$ such that
\begin{equation}\label{drop}
\log \mu_a^k(e^t)\le t_{n+k-\ell}.
\end{equation}

Now let
\begin{equation}\label{Phi-def}
\phi(t)=\log M(e^t,f),\;\;t\in \R.
\end{equation}
Then, by~\eqref{f-cond}, we have
\[
\Phi(t)\ge \phi(t)\ge \left(\frac{\beta}{b}\right)\Phi(t)>\beta\Phi(t)=\Psi_{\beta}(t), \;\;\text{for }t\ge t_N.
\]
Thus if we put
\begin{equation}\label{Tn-def}
T_n=\phi^n(t_{N}),\;\;n\ge 0,
\end{equation}
then $\Phi^n(t_{N})\ge T_n\ge\Psi_{\beta}^n(t_{N})$, for $n\ge 0$, so it follows by~\eqref{nu-n} that
\begin{equation}\label{nodrop}
T_n\in (t_{N+n-1},t_{N+n}],\;\;\text{for }n\ge 0.
\end{equation}
Thus, for any $n\ge N$ and $t\in [t_n,t_{n+1}]$ we can take $\ell=n+1-N$ in~\eqref{drop} and deduce that there exists $k\in\N$ such that
\[
\log \mu_a^k(e^t)\le t_{N+k-1}<T_k,
\]
by~\eqref{nodrop}. Hence, for any sufficiently large $t$, there exists $k\in\N$ such that
\[
\mu_a^k(e^t)<\exp(T_k)=M^k(e^{T_N}),
\]
by~\eqref{Phi-def} and~\eqref{Tn-def}, and it follows that $f$ is not $a$-regular.
\end{proof}

\end{document}